\UseRawInputEncoding
\documentclass[pdflatex,sn-mathphys-num]{sn-jnl}

\UseRawInputEncoding
\usepackage{graphicx}%
\usepackage{multirow}%
\usepackage{amsmath,amssymb,amsfonts}%
\usepackage{amsthm}%
\usepackage{mathrsfs}%
\usepackage[title]{appendix}%
\usepackage{xcolor}%
\usepackage{textcomp}%
\usepackage{manyfoot}%
\usepackage{booktabs}%
\usepackage{algorithm}%
\usepackage{algorithmicx}%
\usepackage{algpseudocode}%
\usepackage{listings}%
\usepackage{inputenc}

\theoremstyle{thmstyleone}%
\newtheorem{theorem}{Theorem}
\newtheorem{proposition}[theorem]{Proposition}%

\theoremstyle{thmstyletwo}%
\newtheorem{example}{Example}%
\newtheorem{remark}{Remark}%
\newtheorem{lemma}{Lemma}
\newtheorem{property}{Property}
\newtheorem{corollary}{Corollary}
\theoremstyle{thmstylethree}%
\newtheorem{definition}{Definition}%

\begin{document}

\title[A generalization of the   $(\alpha, \beta, \gamma)$  von Neumann-Jordan type constant in Banach spaces]{A generalization of the  $(\alpha, \beta, \gamma)$ von Neumann-Jordan type constant in Banach spaces}


\author[1]{\fnm{Shuoyi} \sur{Wan}}\email{170724067@stu.aqnu.edu.cn}

\author[1]{\fnm{Pingru} \sur{Duan}}\email{060125101@stu.aqnu.edu.cn}

\author[1]{\fnm{Zhiyao} \sur{Fang}}\email{fzy73@stu.aqnu.edu.cn}

\author[1]{\fnm{Yuxin} \sur{Wang}}\email{y24060028@stu.aqnu.edu.cn}

\author[1]{\fnm{Qi} \sur{Liu$^*$}}\email{liuq67@aqnu.edu.cn}

\affil*[1]{\orgdiv{School of Mathematics and Statistics}, \orgname{Anqing Normal University}, \orgaddress{ \city{Anqing}, \postcode{246133}, \state{Anhui}, \country{China}}}


\abstract{In this paper,motivated by the pioneering work of Moslehian and Rassias, we introduce a novel generalized von Neumann-Jordan type constant in Banach spaces, a constant that notably encompasses a range of classical geometric constants. First, we establish the bounds of this constant within Banach spaces. Next, we derive its relationships with other geometric constants. Furthermore, we characterize the uniformly non-square property and provide a sufficient condition for the uniform normal structure by leveraging this newly proposed constant.}

\keywords{Banach spaces, geometric constant, von Neumann-Jordan type constant, normal structure}


\pacs[MSC Classification]{46B20}

\maketitle

\section{Introduction}

Geometric constants have drawn growing interest and extensive investigation among mathematicians, owing to their distinctive structural properties and adaptability in addressing a wide range of mathematical challenges. These include the study of isomorphisms of $C_0(K, X)$ spaces \cite{09}, extensions of the Banach-Stone theorem \cite{10}, analysis of the uniform approximation property \cite{13}, and Tingley's problem \cite{14}. For further applications of geometric constants in other mathematical disciplines, we refer the reader to \cite{11,12}.

Let $X=(X,\|\cdot\|)$ be a Banach space. In this article, we assume that the Banach space under consideration is nontrivial, i.e., $\operatorname{dim} X \geq 2$. The geometric structure of a Banach space X is fundamentally characterized by two central objects: its closed unit ball $B_X = \{x \in X : \|x\| \leq 1\}$ and its unit sphere $S_X = \{x \in X : \|x\| = 1\}$.

Subsequently, we recall several key definitions and preliminary results.

\begin{definition}
	\cite{20}  A Banach space $X$ is said to be uniformly non-square if there exists a constant $\delta > 0$ such that for all $x, y \in S_X$ denotes the unit sphere of $X$, the following holds:
	\[\left\| \frac{1}{2}(x + y) \right\| < 1 - \delta \quad \text{or} \quad \left\| \frac{1}{2}(x - y) \right\| < 1 - \delta.\]
\end{definition}

\begin{definition}
	\cite{29,30}  A Banach space  $X$  is said to have normal structure if for every bounded closed convex subset $K$  of  $X$  with more than one point, there exists some  $x \in K$  such that
	\[\sup\{\|x - y\|: y \in K\} < \text{diam}\, K.\]
\end{definition}

\begin{definition}
	\cite{29,30}  $X$  is said to have uniform normal structure if there exists  $0 < c < 1$  such that, for any such subset  $K$ , there is an  $x \in K$  satisfying
	\[\sup\{\|x - y\|: y \in K\} < c \cdot \text{diam}\, K.\]
\end{definition}

 In reference \cite{25},Moslehian and Rassias established a new equivalent characterization of inner product spaces.
 \begin{theorem}\cite{25}
 	Let $(X, \|\cdot\|)$ be a normed space. Then $X$ is an inner product space if and only if for all nonnegative real numbers $\alpha, \beta$ and all $x, y \in X$,
 	
 	$$\|\alpha x + \beta y\|^2 + \|\beta x - \alpha y\|^2 = (\alpha^2 + \beta^2)\bigl(\|x\|^2 + \|y\|^2\bigr).$$
 \end{theorem}

Let us briefly outline the content about the filter.

\begin{definition}\cite{38,39}
We define an ultrafilter $\mathcal{U}$ on an index set  $I$. Let  $\{x_i\}_{i \in I}$  be a subset of a Hausdorff topological space $X$; if  $\{x_i\}_{i \in I}$  converges to  $x$  with respect to  $\mathcal{U}$, we denote this as $\lim_{\mathcal{U}} x_i = x$. Consider a family  $\{X_i\}_{i \in I}$ of Banach spaces. The ultraproduct of $\{X_i\}_{i \in I}$ is the quotient space  $l_\infty(I, X_i)/N_{\mathcal{U}}$, endowed with the quotient norm. Here,

\[l_\infty(I, X_i) = \left\{ (x_i) \mid \|(x_i)\| = \sup_{i \in I} \|x_i\| < \infty \right\},\]

and
\[N_{\mathcal{U}} = \left\{ (x_i) \in l_\infty(I, X_i) \mid \lim_{\mathcal{U}} \|x_i\| = 0 \right\}.\]
\end{definition}

The definition of the James constant is presented below
\[J(X) = \sup \left\{ \min\left\{ \|x + y\|, \|x - y\| \right\} : x, y \in S_X \right\}.\]
For specific results of this constants, please refer to reference \cite{06,07,08}.

The von Neumann-Jordan constant has a long-standing and distinguished history, tracing its origins back to Clarkson \cite{01}, who noted the von Neumann-Jordan as the smallest $C$ which is satisfied the following inequality,
\[\frac{1}{C} \leq \frac{\|x + y\|^2 + \|x - y\|^2}{2(\|x\|^2 + \|y\|^2)} \leq C\]
for all \( x, y \in X \) where \(  (x, y) \neq(0,0) \).\\
For more details, please refer to reference \cite{42}.

Furthermore, the research results presented in reference \cite{01} indicate that this inequality is equivalent to the definition of the von Neumann-Jordan constant, which is given by
	\[C_{\mathrm{NJ}}(X)=\sup \bigg\{\frac{\|x+y\|^{2}+\|x-y\|^{2}}{2\left(\|x\|^{2}+\|y\|^{2}\right)}: x, y \in X, (x, y) \neq(0,0)\bigg\}.\]
An equivalent form of this constant is

\[C_{\mathrm{NJ}}(X) = \sup\left\{ \frac{\|x+y\|^2 + \|x-y\|^2}{4} : x,y \in S_X \right\}.\]

We compile several properties of these constants (refer to \cite{22,23,24} ):
\begin{property}
\begin{enumerate}
	\item[(i)] $\sqrt{2} \leq J(X) \leq 2$  and  $1 \leq C_{\mathrm{NJ}}(X) \leq 2 .$
	\item[(ii)]$C_{\mathrm{NJ}}(X) = C_{\mathrm{NJ}}(X^*) .$
	\item[(iii)] $X$  is uniformly non-square if and only if  $J(X) < 2$  if and only if $C_{\mathrm{NJ}}(X) < 2 .$
	\item[(iv)]  $C_{\mathrm{NJ}}(X) \leq J(X) .$
\end{enumerate}

\end{property}
In recent years, a large number of papers on $p-th$ geometric constants have emerged, which have made meaningful generalizations of classical geometric constants by means of the $p-th$ power and parameterization.

Building on the extension to the \( p \)-th power case, Cui et al. \cite{03} introduced a novel geometric constant \( C_{\mathrm{NJ}}^{p}(X) \), referred to as the generalized von Neumann-Jordan type constant. It is defined as follows:
\begin{definition}
\[
C_{\mathrm{NJ}}^{p}(X)=\sup\left\{ \frac{\|x + y\|^p+\|x - y\|^p}{2^{p - 1}\left(\|x\|^p+\|y\|^p\right)} : x,y\in X,(x,y)\neq(0,0) \right\}
\]
where \( 1\leq p<\infty \). 
\end{definition}
It is worth mentioning that the $p-th$ geometric constant plays an important role in addressing problems related to Banach spaces, such as when dealing with the roundness properties of Banach spaces.\cite{41}

In  new research, Ni et al. \cite{02} introduced a constant with a skew relationship, defined
by 
\[C_{\mathrm{NJ}}^p(\alpha, \beta, X)=\sup\left\{\frac{\left\|\alpha x+\beta y\right\|^p+\|\beta x-\alpha y\|^p}{2^{p-2} \left(\alpha^p+\beta^p\right)\left(\|x\|^p+\|y\|^p\right)}: x, y \in X,(x, y)\neq(0,0)\right\}\]
for $\alpha, \beta > 0, p\geq 1$.	

Then, Wang et al. \cite{28} have defined 
\[\widetilde{C}_{\mathrm{NJ}}^p(\alpha, \beta, X) = \sup\left\{ \frac{\|\alpha x + \beta y\|^{p} + \|\beta x - \alpha y\|^{p}}{2^{p - 1}\left(\alpha^{p} + \beta^{p}\right)} : x, y \in S_{X} \right\}\]
for $\alpha, \beta > 0, p\geq 1$.	

Related to three parameters, Wan et al.\cite{18} have defined 
$$
\begin{aligned}
	C_{\mathrm{NJ}}(\alpha,\beta,\gamma,X)=\sup\left\{\frac{\alpha\|\beta x+\gamma y\|^2+\beta\|\gamma x-\alpha y\|^2}{(\alpha\beta+\gamma^2)(\beta\|x\|^2+\alpha\|y\|^2)}:x,y\in X, (x,y)\neq(0,0)\right\}
\end{aligned}$$ 
for $\alpha,\beta,\gamma>0$.

In 2002, Zb\u{a}ganu \cite{31}  introduce the following constant
\[C_{\mathrm{Z}} = \sup\left\{ \frac{\|x + y\|\|x - y\|}{\|x\|^2 + \|y\|^2} : x, y \in X, (x, y) \neq (0, 0) \right\}.\]

Additionally, the Zb\u{a}ganu constant can also be generalized to the $p-th$ order and in terms of parameters in the following ways:
\begin{enumerate}
	\item[(i)] For \( p \geq 1 \),
	\[ C_{\mathrm{Z}}^{(p)}(X) = \sup\left\{ \frac{\|x+y\|^{\frac{p}{2}} \|x-y\|^{\frac{p}{2}}}{2^{p-2}\left(\|x\|^p + \|y\|^p\right)} : x, y \in X, (x, y) \neq (0, 0) \right\}. \]
	\item[(ii)] For \( \alpha, \beta > 0 \) and \( p \geq 1 \),
	\[C_{\mathrm{Z}}^{(p)}(\alpha, \beta, X) = \sup\left\{ \frac{\|\alpha x + \beta y\|^{\frac{p}{2}} \|\beta x - \alpha y\|^{\frac{p}{2}}}{2^{p - 3}(\alpha^p + \beta^p)\left(\|x\|^p + \|y\|^p\right)} : x, y \in X, (x, y) \neq (0, 0) \right\}.\]
\end{enumerate}
For more details about the above constants, please refer the readers to \cite{08,26,27}.

By observing the constants in the aforementioned generalized forms, this article introduces a  generalized von Neumann-Jordan type constant $C_{\mathrm{NJ}}^p(\alpha, \beta, \gamma, X)$. The arrangement of this article is as follows:

In the second part, we will provid the definition of the new constant $C_{\mathrm{NJ}}^p(\alpha, \beta, \gamma, X)$, and its bounds.

In the third part, we shall establish the relationships between the constant $C_{\mathrm{NJ}}^p(\alpha, \beta, \gamma, X)$ and other constants, including $\delta_X(\varepsilon)$, $J(X)$, and $C_{\mathrm{NJ}}(X)$.

In the forth part, we will derive a sufficient condition for uniform normal structure with respect to the constant $C_{\mathrm{NJ}}^p(\alpha, \beta, \gamma, X)$.

\section{The constant $C_{\mathrm{NJ}}^p(\alpha, \beta, \gamma, X)$}
We introduce a new constant related to three parameters of the $p-th$ order. We begin by introducing the following key definition:

\begin{definition}
	For $\alpha,\beta,\gamma >0$ and $p \geq 1$, we define
	\[
	C_{\mathrm{NJ}}^p(\alpha, \beta, \gamma, X)=\sup \left\{\frac{\alpha^{\frac{p}{2}}\|\beta x+\gamma y\|^p+\beta^{\frac{p}{2}}\|\gamma x-\alpha y\|^p}{2^{p-2}\left[(\alpha \beta)^{\frac{p}{2}}+\gamma^p\right]\left(\beta^{\frac{p}{2}}\|x\|^p+\alpha^{\frac{p}{2}}\|y\|^p\right)}: x, y \in X,(x, y) \neq(0,0)\right\} 
	\]
\end{definition}

\begin{remark}Specially,\\
	(i) If  $p=2$, then, 
	\[
	\begin{aligned}
		C_{\mathrm{NJ}}^2(\alpha, \beta, \gamma, X)= &\sup \left\{\frac{\alpha\|\beta x+\gamma y\|^2+\beta\|\gamma x-\alpha y\|^2}{(\alpha \beta+\gamma^2)\left(\beta\|x\|^2+\alpha\|y\|^2\right)}: x, y \in X,(x, y) \neq(0,0)\right\} \\
		=&C_{\mathrm{NJ}}(\alpha, \beta, \gamma, X).
	\end{aligned}
	\]
	(ii) 
	\[
	\begin{aligned}
		C_{\mathrm{NJ}}^p(\alpha, \alpha, \beta, X)=&\sup \left\{\frac{\alpha^{\frac{p}{2}}\left\| \alpha x+\beta y\right\|^p+\alpha^{\frac{p}{2}}\|\beta x-\alpha y\|^p}{2^{p-2}\alpha^{\frac{p}{2}} \left(\alpha^p+\beta^p\right) \left(\|x\|^p+\|y\|^p\right)}: x, y \in X,(x, y) \neq(0,0)\right\}\\
		=&\sup \left\{\frac{\left\| \alpha x+\beta y\right\|^p+\|\beta x-\alpha y\|^p}{2^{p-2} \left(\alpha^p+\beta^p\right) \left(\|x\|^p+\|y\|^p\right)}: x, y \in X,(x, y) \neq(0,0)\right\}\\
		=&C_{\mathrm{NJ}}^p(\alpha, \beta, X) .
	\end{aligned}
	\]
	(iii) If $\alpha=\beta=\gamma$, then,
	\[
	\begin{aligned}
		C_{\mathrm{NJ}}^p(\alpha, \alpha, \alpha, X)=&\sup \left\{\frac{\alpha^{\frac{p}{2}}\|\alpha x+\alpha y\|^p+\alpha^{\frac{p}{2}}\|\alpha x-\alpha y\|^p}{2^{p-2}\left(\alpha^p+\alpha^p\right)\left(\alpha^{\frac{p}{2}}\|x\|^p+\alpha^{\frac{p}{2}}\|y\|^p\right)}: x, y \in X,(x, y) \neq(0,0)\right\}\\
		=&\sup \left\{\frac{\|x+y\|^p+\|x-y\|^p}{2^{p-1}\left(\|x\|^p+\|y\|^p\right)}: x, y \in X,(x, y) \neq(0,0)\right\}\\
		=&C_{\mathrm{NJ}}^p(X).
	\end{aligned}
	\]
	(iv)If $\alpha=\beta=\gamma=1$, $p=2$, then,
	\[
	C_{\mathrm{NJ}}^2(1, 1, 1, X)=\sup \left\{\frac{\|x+y\|^2+\|x- y\|^2}{2\left(\|x\|^2+\|y\|^2\right)}: x, y \in X,(x, y) \neq(0,0)\right\}=C_{\mathrm{NJ}}(X).
	\]
\end{remark}

Next, we discuss about the lower and upper bounds of the constant $C_{\mathrm{NJ}}^p(\alpha, \beta, \gamma, X)$.

\begin{lemma}\label{le1}
	Let \( x, y > 0 \), \( p \geq 1\), then \( (x + y)^p \leq 2^{p - 1}(x^p + y^p) \).
\end{lemma}

\begin{proposition}\label{p1}
	Suppose that \(X\) is a Banach space, then,
	\[
	\frac{1}{2^{p - 2}}\leq C_{\mathrm{NJ}}^p(\alpha, \beta, \gamma, X) \leq 2.
	\]
\end{proposition}

\begin{proof}
	On the one hand, let $x \neq 0, y=0$, then, 
	\[
	\frac{\alpha^{\frac{p}{2}} \|\beta x\|^p + \beta^{\frac{p}{2}} \|\gamma x\|^p}{2^{p - 2} \beta^{\frac{p}{2}}\left[ (\alpha \beta)^{\frac{p}{2}} + \gamma^p \right]  \|x\|^p} = \frac{\alpha^{\frac{p}{2}} \beta^p + \beta^{\frac{p}{2}} \gamma^p}{2^{p - 2} \beta^{\frac{p}{2}}\left[ (\alpha \beta)^{\frac{p}{2}} + \gamma^p \right]}=\frac{1}{2^{p - 2}}.
	\]
	Hence, 
	\[C_{\mathrm{NJ}}^p(\alpha, \beta, \gamma, X) \geq \frac{1}{2^{p - 2}}.\]
	On the other hand, 
	\[
	\begin{aligned}
		&\frac{\alpha^{\frac{p}{2}} \|\beta x + \gamma y\|^p + \beta^{\frac{p}{2}} \|\gamma x - \alpha y\|^p}{2^{p-2} \left[ (\alpha\beta)^{\frac{p}{2}} + \gamma^p \right] \left( \beta^{\frac{p}{2}} \|x\|^p + \alpha^{\frac{p}{2}} \|y\|^p \right)}\\
		\stackrel{\text{Lamma \ref{le1}}}{\leq}& \frac{2^{p-1} \left[ \alpha^{\frac{p}{2}} \left( \beta^p \|x\|^p + \gamma^p \|y\|^p \right) + \beta^{\frac{p}{2}} \left( \gamma^p \|x\|^p + \alpha^p \|y\|^p \right) \right]}{2^{p-2} \left[ (\alpha\beta)^{\frac{p}{2}} + \gamma^p \right] \left( \beta^{\frac{p}{2}} \|x\|^p + \alpha^{\frac{p}{2}} \|y\|^p \right)}\\
		=&2.
	\end{aligned}
	\]
	This completes the proof.  
\end{proof}
\begin{remark}
	Let $p=1$ , then, $2\leq C_{\mathrm{NJ}}^p(\alpha, \beta, \gamma, X) \leq 2$, we have $C_{\mathrm{NJ}}(\alpha, \beta, \gamma, X)=2$.
\end{remark}
Clearly, $C_{\mathrm{NJ}}^p(\alpha, \beta, \gamma, X)$  also can be rewritten as the following form:
$$C_{\mathrm{NJ}}^p(\alpha, \beta, \gamma, X)=\Bigg\{\frac{\alpha^{\frac{p}{2}}\|\beta x+\gamma t y\|^p+\beta^{\frac{p}{2}}\|\gamma x-\alpha t y\|^p}{2^{p-2}\left[\alpha^{\frac{p}{2}} (\beta^p+\gamma^pt^p)+\beta^{\frac{p}{2}} (\gamma^p+\alpha^pt^p)\right]}: x, y \in S_X, 0 \leq t \leq 1\Bigg\} .$$

\begin{example}
	Let \( X = (\mathbb{R}^2, \|\cdot\|_1) \), then \( C_{\mathrm{NJ}}^p(2, 2, 2, l_1) = 2\).
\end{example}
\begin{proof}
	Let \( x = (0, 1) \), \( y = (1, 0) \), we obtain that
	\[
	\begin{aligned}
		C_{\mathrm{NJ}}^p(\alpha, \beta, \gamma, l_1) &\geq \frac{\alpha^{\frac{p}{2}}\|\beta x+\gamma y\|_1^p+\beta^{\frac{p}{2}}\|\gamma x-\alpha y\|_1^p}{2^{p-2}\left[(\alpha \beta)^{\frac{p}{2}}+\gamma^p\right]\left(\beta^{\frac{p}{2}}\|x\|_1^p+\alpha^{\frac{p}{2}}\|y\|_1^p\right)}\\
		&= \frac{\alpha^{\frac{p}{2}}(\beta+\gamma)^p + \beta^{\frac{p}{2}}(\alpha + \gamma)^p}{2^{p-2}\left[(\alpha\gamma)^{\frac{p}{2}} + \gamma^p\right]\left(\alpha^{\frac{p}{2}}+ \beta^{\frac{p}{2}} \right)}.
	\end{aligned}
	\]
	Now, let $\alpha=\beta=\gamma=2$, then 
	\[C_{\mathrm{NJ}}^p(2, 2, 2, l_1)= \frac{2^{\frac{p}{2}}4^p + 2^{\frac{p}{2}}4^p}{2^{p-2}\left(4^{\frac{p}{2}} + 2^p\right)\left(2^{\frac{p}{2}}+ 2^{\frac{p}{2}} \right)}=2.\]
\end{proof}

\begin{example}
	Let $\alpha,\beta,\gamma>0$, $p>1$ and $t\in[0,1]$, suppose that 
	$$
	\begin{aligned}
		&\gamma^2 (\beta + \gamma t)^{p - 2} \left( \alpha^{\frac{p}{2}} (\beta^p + \gamma^p t^p) + \beta^{\frac{p}{2}} (\gamma^p + \alpha^p t^p) \right) \\<& t^{p - 2} \left( \gamma^p +  \alpha^{\frac{p}{2}}\beta^{\frac{p}{2}} \right) \left( \beta^{\frac{p}{2}} \gamma^p + \alpha^{\frac{p}{2}} (\beta + \gamma t)^p \right)	
	\end{aligned}
	$$
	and
	$$
	\begin{aligned}	 
		&\gamma(\beta + \gamma)^{p - 1} \left( \alpha^{\frac{p}{2}}\beta^{p} +\alpha^{p}\beta^{\frac{p}{2}}+ \alpha^{\frac{p}{2}}\gamma^{p} + \beta^{\frac{p}{2}}\gamma^{p}  \right)\\ -& \left( \beta^{\frac{p}{2}}\gamma^{p} + \alpha^{\frac{p}{2}}(\beta + \gamma)^{p} \right) \left( \alpha^{\frac{p}{2}}\beta^{\frac{p}{2}}+\gamma^{p} \right)<0,
	\end{aligned}
	$$ 
	$X$ be the space $\mathbb{R}^2$ with
	the norm defined by
	\[
	\|x\| =
	\begin{cases} 
		\|x\|_1=|x_1| + |x_2| , & x_1x_2 \leq 0, \\
		\|x\|_\infty=\max\{|x_1|,|x_2|\} , &x_1x_2\geq0. 
	\end{cases}
	\] 
	Then$$C_{\mathrm{NJ}}^p(\alpha, \beta, \gamma, X)=\frac{\beta^\frac{p}{2}\gamma^p+\alpha^\frac{p}{2}(\beta+\gamma t)^p}{2^{p-2}\left[\alpha^{\frac{p}{2}} (\beta^p+\gamma^pt^p)+\beta^{\frac{p}{2}} (\gamma^p+\alpha^pt^p)\right]},$$
	where $t_0\in(0,1)$ is the only solution of the equation $$\gamma(\beta + \gamma)^{p - 1} \left( \alpha^{\frac{p}{2}}\beta^{p} + \alpha^{\frac{p}{2}}\gamma^{p} + \beta^{\frac{p}{2}}\gamma^{p} + \beta^{\frac{p}{2}}\alpha^{p} \right) = \left( \beta^{\frac{p}{2}}\gamma^{p} + \alpha^{\frac{p}{2}}(\beta + \gamma)^{p} \right) \left( \gamma^{p} + \beta^{\frac{p}{2}}\alpha^{\frac{p}{2}} \right).$$
	
\end{example}
\begin{proof}
	By Minkowski inequality, for any $\lambda, \mu \in[0,1]$ and any $x_1, x_2, y_1, y_2 \in B_X$ with $x=\lambda x_1+(1- \lambda) x_2,~y=\mu y_2+(1-\mu) y_2$, we have
	$$
	\begin{aligned}
		& \alpha^{\frac{p}{2}}\|\beta x+\gamma t y\|^p+\beta^{\frac{p}{2}}\|\gamma x-\alpha t y\|^p \\
		=&\alpha^{\frac{p}{2}}\left\|\lambda\left(\beta x_1+\gamma t y\right)+(1-\lambda)\left(\beta x_2+\gamma t y\right)\right\|^p\\ +&\beta^{\frac{p}{2}}\left\|\lambda\left(\gamma x_1-\alpha t y\right)+(1-\lambda)\left(\gamma x_2-\alpha t y\right)\right\|^p \\
		\leq& \alpha^{\frac{p}{2}}\left[\lambda\left\|\beta x_1+\gamma t y\right\|^p+(1-\lambda)\left\|\beta x_2+ \gamma t y\right\|^p\right] \\ +&\beta^{\frac{p}{2}}\left[\lambda\left\|\gamma x_1-\alpha t y\right\|^p+(1-\lambda)\left\|\gamma x_2-\alpha t y\right\|^p\right] \\
		\leq &\lambda \mu\left[\alpha^{\frac{p}{2}}\left\|\beta x_1+\gamma t y_1\right\|^p+\beta^{\frac{p}{2}}\left\|\gamma x_1-\alpha t y_1\right\|^p\right]\\+&\lambda(1-\mu)\left[\alpha^{\frac{p}{2}}\left\|\beta x_1+\gamma t y_2\right\|^p+\beta^{\frac{p}{2}}\left\|\gamma x_1-\alpha t y_2\right\|^p\right] \\
		+&(1-\lambda)\mu\left[\alpha^{\frac{p}{2}}\left\|\beta x_2+\gamma t y_1\right\|^p+\beta^{\frac{p}{2}}\left\|\gamma x_2-\alpha t y_1\right\|^p\right]\\+&(1-\lambda)(1-\mu)\left[\alpha^{\frac{p}{2}}\left\|\beta x_2+\gamma t y_2\right\|^p+\beta^{\frac{p}{2}}\left\|\gamma x_2-\alpha t y_2\right\|^p\right].
	\end{aligned}
	$$
	Thus, $$
	\begin{aligned}
		\alpha^{\frac{p}{2}}\|\beta x+\gamma t y\|^p+\beta^{\frac{p}{2}}\|\gamma x-\alpha t y\|^p 
		\leq \max \bigg\{&\alpha^{\frac{p}{2}}\left\|\beta x_1+\gamma t y_1\right\|^p+\beta^{\frac{p}{2}}\left\|\gamma x_1-\alpha t y_1\right\|^p,\\&\alpha^{\frac{p}{2}}\left\|\beta x_1+\gamma t y_2\right\|^p+\beta^{\frac{p}{2}}\left\|\gamma x_1-\alpha t y_2\right\|^p,\\
		&\alpha^{\frac{p}{2}}\left\|\beta x_2+\gamma t y_1\right\|^p+\beta^{\frac{p}{2}}\left\|\gamma x_2-\alpha t y_1\right\|^p, \\&\alpha^{\frac{p}{2}}\left\|\beta x_2+\gamma t y_2\right\|^p+\beta^{\frac{p}{2}}\left\|\gamma x_2-\alpha t y_2\right\|^p\bigg\}.
	\end{aligned}
	$$
	Then, to determine the values of the constant  $C_{\mathrm{NJ}}^p(\alpha, \beta, \gamma, X)$ , we only need to examine its behavior at the extreme points of the unit ball.
	$$\operatorname{ex}\left(B_X\right) = \big\{(1,0), (1,1), (0,1), (-1,0), (-1,-1), (0,-1)\big\}.$$
	
	Hence, we only need to prove $\alpha^\frac{p}{2}\|\beta x+\gamma t y\|^p+\beta^\frac{p}{2}\|\gamma x-\alpha t y\|^p \leq \beta^\frac{p}{2}\gamma^p+\alpha^\frac{p}{2}(\beta+\gamma t)^p$ for any $x, y \in \operatorname{ex}\left(B_X\right)$ and every $t \in[0,1]$.
	For these $x, y$, by some calculations, we  have $$\alpha^\frac{p}{2}\|\beta x+\gamma t y\|^p+\beta^\frac{p}{2}\|\gamma x-\alpha t y\|^p \leq \beta^\frac{p}{2}\gamma^p+\alpha^\frac{p}{2}(\beta+\gamma t)^p$$ for every $t \in[0,1]$. Therefore,$$C_{\mathrm{N J}}^p(\alpha,\beta,\gamma,X)
	\leq\sup_{t\in[0,1]}\left\{\frac{\beta^\frac{p}{2}\gamma^p+\alpha^\frac{p}{2}(\beta+\gamma t)^p}{2^{p-2}\left[\alpha^{\frac{p}{2}} (\beta^p+\gamma^pt^p)+\beta^{\frac{p}{2}} (\gamma^p+\alpha^pt^p)\right]}\right\}.$$
	Now we consider the function
	$$f(t)=\frac{\beta^\frac{p}{2}\gamma^p+\alpha^\frac{p}{2}(\beta+\gamma t)^p}{\alpha^{\frac{p}{2}} (\beta^p+\gamma^pt^p)+\beta^{\frac{p}{2}} (\gamma^p+\alpha^pt^p)},$$
	then, 
	$$f'(t) = \frac{\alpha^{\frac{p}{2}}p}{\left[ \alpha^{\frac{p}{2}} (\beta^p + \gamma^p t^p) + \beta^{\frac{p}{2}} (\gamma^p + \alpha^p t^p) \right]^2} h(t),$$
	where 
	$$
	\begin{aligned}
		h(t)=& \gamma (\beta + \gamma t)^{p-1} \cdot \left[ \alpha^{\frac{p}{2}}(\beta^{p} + \gamma^{p} t^{p}) + \beta^{\frac{p}{2}}(\gamma^{p} + \alpha^{p} t^{p}) \right]\\ -& t^{p-1} \left( \beta^{\frac{p}{2}}\gamma^p + \alpha^{\frac{p}{2}}(\beta + \gamma t)^p \right) \left( \gamma^p +  \alpha^{\frac{p}{2}}\beta^{\frac{p}{2}} \right),
	\end{aligned}
	$$ 
	then 
	$$
	\begin{aligned}
		h'(t) &= (p - 1) \bigg[ \gamma^2 (\beta + \gamma t)^{p - 2} \left( \alpha^{\frac{p}{2}} (\beta^p + \gamma^p t^p) + \beta^{\frac{p}{2}} (\gamma^p + \alpha^p t^p) \right)\\ &- t^{p - 2} \left( \gamma^p + \alpha^{\frac{p}{2}}\beta^{\frac{p}{2}}  \right) \left( \beta^{\frac{p}{2}} \gamma^p + \alpha^{\frac{p}{2}} (\beta + \gamma t)^p \right) \bigg] <0,
	\end{aligned}
	$$
	thus $h(t)$  is decreasing from $$\alpha^{\frac{p}{2}}  \beta^{2p - 1}\gamma + \beta^{\frac{3p}{2} - 1} \gamma^{p + 1}>0$$ to $$
	\begin{aligned}	 
		&\gamma(\beta + \gamma)^{p - 1} \left( \alpha^{\frac{p}{2}}\beta^{p} +\alpha^{p}\beta^{\frac{p}{2}}+ \alpha^{\frac{p}{2}}\gamma^{p} + \beta^{\frac{p}{2}}\gamma^{p}  \right)\\ -& \left( \beta^{\frac{p}{2}}\gamma^{p} + \alpha^{\frac{p}{2}}(\beta + \gamma)^{p} \right) \left( \alpha^{\frac{p}{2}}\beta^{\frac{p}{2}}+\gamma^{p} \right)<0.
	\end{aligned}
	$$ 
	Therefore, there is an only $t_0 \in (0, 1)$ such that $g(t_0) = 0$.
	Then, we have$$C_{\mathrm{N J}}^p(\alpha,\beta,\gamma,X)
	\leq\frac{\beta^\frac{p}{2}\gamma^p+\alpha^\frac{p}{2}(\beta+\gamma t_0)^p}{2^{p-2}\left[\alpha^{\frac{p}{2}} (\beta^p+\gamma^pt_0^p)+\beta^{\frac{p}{2}} (\gamma^p+\alpha^pt_0^p)\right]}.$$
	On the other hand, let $x=(1,0), y=(t_0,t_0)$, then we have
	$$\frac{\alpha^{\frac{p}{2}}\|\beta x+\gamma y\|^p+\beta^{\frac{p}{2}}\|\gamma x-\alpha y\|^p}{2^{p-2}\left[(\alpha \beta)^{\frac{p}{2}}+\gamma^p\right]\left(\beta^{\frac{p}{2}}\|x\|^p+\alpha^{\frac{p}{2}}\|y\|^p\right)}=\frac{\beta^\frac{p}{2}\gamma^p+\alpha^\frac{p}{2}(\beta+\gamma t_0)^p}{2^{p-2}\left[\alpha^{\frac{p}{2}} (\beta^p+\gamma^pt_0^p)+\beta^{\frac{p}{2}} (\gamma^p+\alpha^pt_0^p)\right]}.$$ This implies that $$C_{\mathrm{N J}}^p(\alpha,\beta,\gamma,X)\geq\frac{\beta^\frac{p}{2}\gamma^p+\alpha^\frac{p}{2}(\beta+\gamma t_0)^p}{2^{p-2}\left[\alpha^{\frac{p}{2}} (\beta^p+\gamma^pt_0^p)+\beta^{\frac{p}{2}} (\gamma^p+\alpha^pt_0^p)\right]},$$ as desired.
\end{proof}

\section{Relationship with other geometric constants}
In this part, we establish the new constant's connections with the modulus of convexity  $\delta_X(\varepsilon)$, the geometric constant $C_{\mathrm{NJ}}(X)$ and the James constant $J(X)$. Since  $J(X)$ and $C_{\mathrm{NJ}}(X)$ have been introduced in the first part, we now recall the definition of  $\delta_X(\varepsilon)$.
\begin{definition}
	\cite{04,05}Let \((X, \|\cdot\|)\) be a Banach space. For every \(\varepsilon \in [0, 2]\), we define the modulus of convexity of \(\|\cdot\|\) by
	\[\delta_X(\varepsilon) = \inf\left\{ 1 - \frac{\|x + y\|}{2} \,:\, x, y \in S_X,\ \|x - y\| = \varepsilon \right\}. \]
\end{definition}

Then, we establish the relationship between $C_{\mathrm{N J}}^p(\alpha,\beta,\gamma,X)$ and $\delta_X(\varepsilon)$ by restricting $\varepsilon \in (1, 2]$.

\begin{theorem}\label{the1}
	Let \(\beta \leq \gamma\) and \(X\) be a Banach space. If  
	\[C_{\mathrm{N J}}^p(\alpha,\beta,\gamma,X) < \frac{\alpha^\frac{p}{2}(\beta+ \gamma)^p+\beta^\frac{p}{2}(\alpha+\gamma)^p(\varepsilon - 1)^p}{2^{p-2}(\alpha^{\frac{p}{2}} \beta^{\frac{p}{2}}+ \gamma^p)(\alpha^{\frac{p}{2}} + \beta^{\frac{p}{2}})}\]
	for all $p \geq 1$, \(\varepsilon \in (1, 2]\), then \(\delta_X(\varepsilon) > 0\).
\end{theorem}

\begin{proof}
	Assume that \(\delta_X(\varepsilon) = 0\), then there exist \(x_n, y_n \in S_X\) such that \(\|x_n - y_n\| = \varepsilon\) for all \(n \in \mathbb{N}\) and \(\lim_{n \to \infty} \|x_n + y_n\| = 2\). According to the following elementary inequality
	\[
	\begin{aligned}
		\|\beta x_n + \gamma y_n\| &= \| \beta (x_n + y_n) + (\gamma - \beta) y_n \| \\
		& \geq \beta \|x_n + y_n\| - (\gamma - \beta) \|y_n\| \\
		&\to \beta + \gamma \quad (n \to \infty)
	\end{aligned}
	\]
	and
	\[
	\begin{aligned}
		\|\gamma x_n - \alpha y_n\| &= \| (\gamma + \alpha) (x_n - y_n) + \gamma y_n - \alpha x_n \| \\
		& \geq (\gamma + \alpha) \|x_n - y_n\| - \gamma \|y_n\| - \alpha \|x_n\|\\
		&=(\gamma + \alpha) (\varepsilon - 1).
	\end{aligned}
	\]
	Thus, we can deduce that
	\[
	\begin{aligned}
		\frac{\alpha^\frac{p}{2}(\beta+ \gamma)^p+\beta^\frac{p}{2}(\alpha+\gamma)^p(\varepsilon - 1)^p}{2^{p-2}(\alpha^{\frac{p}{2}} \beta^{\frac{p}{2}}+ \gamma^p)(\alpha^{\frac{p}{2}} + \beta^{\frac{p}{2}})} &\leq \liminf_{n \to \infty} \frac{\alpha^{\frac{p}{2}} \|\beta x + \gamma y\|^p + \beta^{\frac{p}{2}} \|\gamma x - \alpha y\|^p}{2^{p - 2} (\alpha^{\frac{p}{2}} \beta^{\frac{p}{2}} + \gamma^p) (\beta^{\frac{p}{2}} \|x_n\|^p +  \alpha^{\frac{p}{2}}\|y_n\|^p)} \\
		& \leq C_{\mathrm{N J}}^p(\alpha,\beta,\gamma,X) \\
		& < \frac{\alpha^\frac{p}{2}(\beta+ \gamma)^p+\beta^\frac{p}{2}(\alpha+\gamma)^p(\varepsilon - 1)^p}{2^{p-2}(\alpha^{\frac{p}{2}} \beta^{\frac{p}{2}}+ \gamma^p)(\alpha^{\frac{p}{2}} + \beta^{\frac{p}{2}})} \\
	\end{aligned}
	\]
	a contradiction, as desired.
\end{proof}

\begin{proposition}\label{prop2}
	Let \(\beta \leq \gamma\) and \(X\) be a Banach space, then
	\[
	\begin{aligned}
		C_{\mathrm{N J}}^p(\alpha,\beta,\gamma,X) \geq \frac{\beta^{\frac{p}{2}} [(\alpha + \gamma)(\varepsilon - 1)]^p + \left( \alpha^{\frac{1}{2}} \beta - 2\alpha^{\frac{1}{2}} \beta \delta_X(\varepsilon) - \alpha^{\frac{1}{2}} \gamma \right)^p}{2^{p - 2} (\alpha^{\frac{p}{2}} \beta^{\frac{p}{2}} + \gamma^p) (\alpha^{\frac{p}{2}} + \beta^{\frac{p}{2}})},
	\end{aligned}
	\]
	where \(p \geq 1\), \(\varepsilon \in (1, 2]\).
\end{proposition}
\begin{proof}
	Since $\delta_X(\varepsilon)$ has an equivalent definition \[\delta_X(\varepsilon) = \inf\left\{ 1 - \frac{\|x + y\|}{2} \,:\, x, y \in B_X,\ \|x - y\| \geq \varepsilon \right\},\quad \varepsilon \in [0, 2],\]  we suppose that there exists \(x, y\in B_X \subset X\) such that \(\|x - y\| \geq \varepsilon\), then
	\[
	\begin{aligned}
		C_{\mathrm{N J}}^p(\alpha,\beta,\gamma,X) &\geq \frac{\alpha^{\frac{p}{2}} \|\beta x + \gamma y\|^p + \beta^{\frac{p}{2}} \|\gamma x - \alpha y\|^p}{2^{p - 2} (\alpha^{\frac{p}{2}} \beta^{\frac{p}{2}} + \gamma^p) (\beta^{\frac{p}{2}} \|x\|^p + \alpha^{\frac{p}{2}} \|y\|^p)}\\
		&\geq \frac{\alpha^{\frac{p}{2}} (\beta \|x_n + y_n\| - (\gamma - \beta) \|y_n\|)^p + \beta^{\frac{p}{2}} [(\alpha+\gamma) (\varepsilon - 1)]^p}{2^{p - 2} (\alpha^{\frac{p}{2}} \beta^{\frac{p}{2}} + \gamma^p) (\beta^{\frac{p}{2}} \|x\|^p + \alpha^{\frac{p}{2}} \|y\|^p)} .
	\end{aligned}
	\]
	Thus, 
	\[
	\begin{aligned}
		& 1 + \frac{\alpha^{\frac{1}{2}} \beta - \alpha^{\frac{1}{2}} \gamma - (C_{\mathrm{N J}}^p(\alpha,\beta,\gamma,X)  \cdot 2^{p - 2} (\alpha^{\frac{p}{2}} \beta^{\frac{p}{2}} + \gamma^p) (\alpha^{\frac{p}{2}} + \beta^{\frac{p}{2}}) - \beta^{\frac{p}{2}} [(\gamma + \alpha) (\varepsilon - 1)^p])^{\frac{1}{p}}}{2 \alpha^{\frac{1}{2}} \beta}\\\leq& 1 - \frac{\|x + y\|}{2}.
	\end{aligned}
	\]
	Hence, 
	\[
	\begin{aligned}
		\frac{1}{2} - \frac{\alpha^{\frac{1}{2}} \gamma + (C_{\mathrm{N J}}^p(\alpha,\beta,\gamma,X) \cdot 2^{p - 2} (\alpha^{\frac{p}{2}} \beta^{\frac{p}{2}} + \gamma^p) (\alpha^{\frac{p}{2}} + \beta^{\frac{p}{2}}) - \beta^{\frac{p}{2}} [(\alpha+\gamma) (\varepsilon - 1)^p])^{\frac{1}{p}}}{2 \alpha^{\frac{1}{2}} \beta} \leq \delta_X (\varepsilon).
	\end{aligned}
	\]
	So that, we have
	\[
	\begin{aligned}
		C_{\mathrm{N J}}^p(\alpha,\beta,\gamma,X) \geq \frac{\beta^{\frac{p}{2}} [(\alpha + \gamma)(\varepsilon - 1)]^p + \left( \alpha^{\frac{1}{2}} \beta - 2\alpha^{\frac{1}{2}} \beta \delta_X(\varepsilon) - \alpha^{\frac{1}{2}} \gamma \right)^p}{2^{p - 2} (\alpha^{\frac{p}{2}} \beta^{\frac{p}{2}} + \gamma^p) (\alpha^{\frac{p}{2}} + \beta^{\frac{p}{2}})}.
	\end{aligned}
	\]
\end{proof}

\begin{remark}
	The restriction $\varepsilon \in (1, 2]$ in Theorem~\ref{the1} and Proposition~\ref{prop2} are imposed to ensure the positivity of $\varepsilon-1.$
\end{remark}

Next, we establish a relationship among \( C_{\mathrm{NJ}}^p(\alpha, \beta, \gamma, X) \), \( C_{\mathrm{NJ}}(X) \), and \( J(X) \).

\begin{theorem}
	Let \( X \) be a Banach space. Then, the following conclusions are equivalent:
	\begin{enumerate}
		\item[\textit{(i)}] \( J(X) = 2 \).
		\item[\textit{(ii)}] \( C_{\mathrm{NJ}}(X) = 2 \).
		\item[\textit{(iii)}] \(C_{\mathrm{NJ}}^p(\alpha, \beta, \gamma, X) = \frac{\alpha^{\frac{p}{2}}(\beta+\gamma)^p+\beta^{\frac{p}{2}}(\alpha+\gamma )^p}{2^{p-2}\left[(\alpha \beta)^{\frac{p}{2}}+\gamma^p\right]\left(\alpha^{\frac{p}{2}}+\beta^{\frac{p}{2}}\right)}  \) .
	\end{enumerate}
\end{theorem}

\begin{proof}
	As established in prior work (see \cite{17}), the conditions \( J(X) = 2 \) and \( C_{\mathrm{NJ}}(X) = 2 \) are equivalent. We now prove the equivalence between \( C_{\mathrm{NJ}}(X) = 2 \) and \(C_{\mathrm{NJ}}^p(\alpha, \beta, \gamma, X) = \frac{\alpha^{\frac{p}{2}}(\beta+\gamma)^p+\beta^{\frac{p}{2}}(\alpha+\gamma )^p}{2^{p-2}\left[(\alpha \beta)^{\frac{p}{2}}+\gamma^p\right]\left(\alpha^{\frac{p}{2}}+\beta^{\frac{p}{2}}\right)}  \).\\
	Suppose \( C_{\mathrm{NJ}}(X) = 2 \), with the equivalent definition of the constant $C_{\mathrm{NJ}}(X)$, this implies the existence of sequences \( \{x_n\} \subseteq S_X \) and \( \{y_n\} \subseteq B_X \) such that
	\begin{equation}\label{eq1}
		\frac{\|x_n + y_n\|^2 + \|x_n - y_n\|^2}{2(1 + \|y_n\|^2)} \to 2 \quad (n \to \infty). 
	\end{equation}
	
	For (\ref{eq1}) to hold, it is necessary that
	\begin{equation}\label{eq2}
		\|x_n + y_n\| \to 2 , \quad \|x_n - y_n\| \to 2 , \quad \text{and} \quad \|y_n\| \to 1 \quad (n \to \infty).
	\end{equation}
	Now we prove that 
	\begin{equation}\label{eq3}
		\frac{\alpha^{\frac{p}{2}}\|\beta x_n+\gamma y_n\|^p+\beta^{\frac{p}{2}}\|\gamma x_n-\alpha y_n\|^p}{2^{p-2}\left[(\alpha \beta)^{\frac{p}{2}}+\gamma^p\right]\left(\beta^{\frac{p}{2}}+\alpha^{\frac{p}{2}}\|y_n\|^p\right)} \to  \frac{\alpha^{\frac{p}{2}}(\beta+\gamma)^p+\beta^{\frac{p}{2}}(\alpha+\gamma )^p}{2^{p-2}\left[(\alpha \beta)^{\frac{p}{2}}+\gamma^p\right]\left(\alpha^{\frac{p}{2}}+\beta^{\frac{p}{2}}\right)} (n \to \infty).
	\end{equation}
	
	It is evident that the constant $C_{\mathrm{NJ}}^p(\alpha, \beta, \gamma, X)$ can be expressed in the following form
	\[C_{\mathrm{NJ}}^p(\alpha, \beta, \gamma, X)=\sup \left\{\frac{\alpha^{\frac{p}{2}}\|\beta x+\gamma y\|^p+\beta^{\frac{p}{2}}\|\gamma x-\alpha y\|^p}{2^{p-2}\left[(\alpha \beta)^{\frac{p}{2}}+\gamma^p\right]\left(\beta^{\frac{p}{2}}+\alpha^{\frac{p}{2}}\|y\|^p\right)}: x \in S_X , y \in B_X \right\}. \]
	We consider the following two cases.\\
	\textbf{Case 1:} If $\alpha = \beta = \gamma,$ then by Proposition~\ref{p1} , we have \[\frac{\alpha^{\frac{p}{2}}\|\beta x+\gamma y\|^p+\beta^{\frac{p}{2}}\|\gamma x-\alpha y\|^p}{2^{p-2}\left[(\alpha\beta)^{\frac{p}{2}}+\gamma^p\right]\left(\beta^{\frac{p}{2}}+\alpha^{\frac{p}{2}}\|y\|^p\right)}=\frac{\|x_n + y_n\|^p + \|x_n - y_n\|^p}{2^{p-1} \left(1 + \|y_n\|^p\right)} \to  2 \quad (n \to \infty).\]Thus, (\ref{eq3}) holds.\\ \textbf{Case 2:} If $\alpha \neq \beta \neq \gamma$, we now consider the situation where
	$\alpha > \beta > \gamma$ without loss of generality (the other situationes is similar).\\
	Since
	\[\begin{aligned}
		&\alpha^{\frac{p}{2}} \|\beta x_n + \gamma y_n\|^p + \beta^{\frac{p}{2}} \|\gamma x_n - \alpha y_n\|^p \\
		=& \alpha^{\frac{p}{2}} \left\| \beta (x_n + y_n) - (\beta - \gamma) y_n \right\|^p + \beta^{\frac{p}{2}} \left\| (\alpha + \gamma) (x_n - y_n) + \gamma y_n -\alpha x_n \right\|^p \\
		\geq& \alpha^{\frac{p}{2}} \left[ \beta \|x_n + y_n\| - (\beta - \gamma) \|y_n\| \right]^p + \beta^{\frac{p}{2}} \left[ (\alpha + \gamma) \|x_n - y_n\| - \gamma \|y_n\| - \alpha \|x_n\| \right]^p\end{aligned}\]
	and
	\[
	\begin{aligned}
		&\alpha^{\frac{p}{2}} \|\beta x_n + \gamma y_n\|^p + \beta^{\frac{p}{2}} \|\gamma x_n - \alpha y_n\|^p \\
		\leq& \alpha^{\frac{p}{2}} \left( \beta \|x_n\| + \gamma \|y_n\| \right)^p + \beta^{\frac{p}{2}} \left\| \gamma (x_n - y_n) + (\gamma - \alpha) y_n \right\|^p \\
		\leq& \alpha^{\frac{p}{2}} \left( \beta \|x_n \| + \gamma \| y_n\| \right)^p + \beta^{\frac{p}{2}} \left( \gamma \|x_n - y_n\| + (\alpha - \gamma) \|y_n\| \right)^p.
	\end{aligned}
	\]
	So, we have 
	\[
	\begin{aligned}
		&\alpha^{\frac{p}{2}} \left[ \beta \|x_n + y_n\| - (\beta - \gamma) \|y_n\| \right]^p + \beta^{\frac{p}{2}} \left[ (\alpha + \gamma) \|x_n - y_n\| - \gamma \|y_n\| - \alpha \|x_n\| \right]^p \\ \leq& \alpha^{\frac{p}{2}} \|\beta x_n + \gamma y_n\|^p + \beta^{\frac{p}{2}} \|\gamma x_n - \alpha y_n\|^p \\ \leq& \alpha^{\frac{p}{2}} \left( \beta \|x_n \| + \gamma \| y_n\| \right)^p + \beta^{\frac{p}{2}} \left( \gamma \|x_n - y_n\| + (\alpha - \gamma) \|y_n\| \right)^p.
	\end{aligned}
	\]
	Combine with (\ref{eq2}), we have 
	\[
	\begin{aligned}
		&\alpha^{\frac{p}{2}} \left[ \beta \|x_n + y_n\| - (\beta - \gamma) \|y_n\| \right]^p + \beta^{\frac{p}{2}} \left[ (\alpha + \gamma) \|x_n - y_n\| - \gamma \|y_n\| - \alpha \|x_n\| \right]^p \\
		\to& \alpha^{\frac{p}{2}} (\beta + \gamma)^p + \beta^{\frac{p}{2}} (\alpha + \gamma)^p \quad (n \to \infty)
	\end{aligned}
	\]
	and 
	\[
	\begin{aligned}
		&\alpha^{\frac{p}{2}} \left( \beta \|x_n \| + \gamma \| y_n\| \right)^p + \beta^{\frac{p}{2}} \left( \gamma \|x_n - y_n\| + (\alpha - \gamma) \|y_n\| \right)^p \\
		\to& \alpha^{\frac{p}{2}} (\beta + \gamma)^p + \beta^{\frac{p}{2}} (\alpha + \gamma)^p \quad (n \to \infty).
	\end{aligned}
	\]
	Thus,
	\[\alpha^{\frac{p}{2}} \|\beta x_n + \gamma y_n\|^p + \beta^{\frac{p}{2}} \|\gamma x_n - \alpha y_n\|^p \to \alpha^{\frac{p}{2}} (\beta + \gamma)^p + \beta^{\frac{p}{2}} (\alpha + \gamma)^p \quad (n \to \infty).\]
	Based on the two cases discussed above, (\ref{eq3}) holds. \\
	Thus, \( C_{\mathrm{NJ}}(X) = 2 \) iff \(C_{\mathrm{NJ}}^p(\alpha, \beta, \gamma, X) = \frac{\alpha^{\frac{p}{2}}(\beta+\gamma)^p+\beta^{\frac{p}{2}}(\alpha+\gamma )^p}{2^{p-2}\left[(\alpha \beta)^{\frac{p}{2}}+\gamma^p\right]\left(\alpha^{\frac{p}{2}}+\beta^{\frac{p}{2}}\right)}  \), as desired.
\end{proof}

Additionally, we observe that the constant  $C_{\mathrm{N J}}^p(\alpha,\beta,\gamma,X)$  exhibits a certain relation to  $J(X)$.

\begin{proposition}\label{prop1}
	Let  \(X\) be a Banach space, then we have 
	\[
	C_{\mathrm{NJ}}^p(\alpha, \beta, \gamma, X) \geq \frac{\max\left\{\left(\alpha^{\frac12}\beta J(X) - |\beta-\gamma|\right)^p,\left(\beta^{\frac12}\gamma J(X) - |\gamma-\alpha|\right)^p\right\}}{2^{p - 3} \left( \alpha^{\frac{p}{2}} \beta^{\frac{p}{2}} + \gamma^p \right) \left( \alpha^{\frac{p}{2}} + \beta^{\frac{p}{2}} \right)}.
	\]
\end{proposition}
\begin{proof}
	For any $x,y \in S_X$, we have 
	\[
	\begin{aligned}
		& 2 \min \left\{ \alpha^{\frac{1}{2}} \|\beta x + \gamma y\|, \beta^{\frac{1}{2}} \|\gamma x - \alpha y\| \right\}^p \\ \leq& \alpha^{\frac{p}{2}} \|\beta x + \gamma y\|^p + \beta^{\frac{p}{2}} \|\gamma x - \alpha y\|^p \\
		\leq& 2^{p - 2} \left( \alpha^{\frac{p}{2}} \beta^{\frac{p}{2}} + \gamma^p \right) \left( \alpha^{\frac{p}{2}} +\beta^{\frac{p}{2}} \right) C_{\mathrm{N J}}^p(\alpha,\beta,\gamma,X), \\
	\end{aligned}
	\]
	so
	\[
	\min \left\{ \alpha^{\frac{1}{2}} \|\beta x + \gamma y\|, \beta^{\frac{1}{2}} \|\gamma x - \alpha y\| \right\} \leq \left[ 2^{p - 3} \left( \alpha^{\frac{p}{2}} \beta^{\frac{p}{2}} + \gamma^p \right) \left( \alpha^{\frac{p}{2}} + \beta^{\frac{p}{2}} \right) C_{\mathrm{N J}}^p(\alpha,\beta,\gamma,X) \right]^{\frac{1}{p}} .
	\]
	Then, we can obtain
	\[
	\begin{aligned}
		& \left[ 2^{p - 3} \left( \alpha^{\frac{p}{2}} \beta^{\frac{p}{2}} + \gamma^p \right) \left( \alpha^{\frac{p}{2}} + \beta^{\frac{p}{2}} \right) C_{\mathrm{N J}}^p(\alpha,\beta,\gamma,X) \right]^{\frac{1}{p}} 
		\\\geq& \min \left\{ \alpha^{\frac{1}{2}} \|\beta x + \gamma y\|, \beta^{\frac{1}{2}} \|\gamma x - \alpha y\| \right\} \\
		\geq& \min \left\{ \alpha^{\frac{1}{2}} (\beta \|x + y\| - |\beta - \gamma|), \beta^{\frac{1}{2}} (\gamma \|x - y\| - |\gamma - \alpha|) \right\}.
	\end{aligned} 
	\]
	Thus, 
	\[
	C_{\mathrm{NJ}}^p(\alpha, \beta, \gamma, X) \geq \frac{\left( \alpha^{\frac{1}{2}} \beta J(X) - |\beta - \gamma| \right)^p}{2^{p - 3} \left( \alpha^{\frac{p}{2}} \beta^{\frac{p}{2}} + \gamma^p \right) \left( \alpha^{\frac{p}{2}} + \beta^{\frac{p}{2}} \right)}.
	\]
	Similarly,
	\[
	C_{\mathrm{NJ}}^p(\alpha, \beta, \gamma, X) \geq \frac{\left( \beta^{\frac{1}{2}} \gamma J(X) - |\gamma - \alpha| \right)^p}{2^{p - 3} \left( \alpha^{\frac{p}{2}} \beta^{\frac{p}{2}} + \gamma^p \right) \left( \alpha^{\frac{p}{2}} + \beta^{\frac{p}{2}} \right)},
	\]
	as desired.
\end{proof}

Now, by combining Proposition~\ref{prop1} with the result that a Banach space $X$
is uniformly non-square if and only if $J(X)<2$ \cite{07}, we can obtain a straightforward corollary.

\begin{corollary}\label{cor1}
	Let $X$ be a Banach space, if \[ C_{\mathrm{NJ}}^p(\alpha, \beta, \gamma, X) < \frac{\max\left\{\left(2\alpha^{\frac12}\beta  - |\beta-\gamma|\right)^p,\left(2\beta^{\frac12}\gamma - |\gamma-\alpha|\right)^p\right\}}{2^{p - 3} \left( \alpha^{\frac{p}{2}} \beta^{\frac{p}{2}} + \gamma^p \right) \left( \alpha^{\frac{p}{2}} + \beta^{\frac{p}{2}} \right)}\] holds for some $\alpha,\beta,\gamma>0$, then $X$ is uniformly non-square.
\end{corollary}

\begin{proof}
	Assume that $X$ is not uniformly non-square. By \cite{07}, we have $J(X) \ge 2$. Conbine with Proposition~\ref{prop1}, we have 
	\[ C_{\mathrm{NJ}}^p(\alpha, \beta, \gamma, X) \geq \frac{\max\left\{\left(2\alpha^{\frac12}\beta  - |\beta-\gamma|\right)^p,\left(2\beta^{\frac12}\gamma - |\gamma-\alpha|\right)^p\right\}}{2^{p - 3} \left( \alpha^{\frac{p}{2}} \beta^{\frac{p}{2}} + \gamma^p \right) \left( \alpha^{\frac{p}{2}} + \beta^{\frac{p}{2}} \right)}\]
	which contradicts the condition of Corollary~\ref{cor1}.
\end{proof}
\section{The constant $C_{\mathrm{NJ}}^p(\alpha, \beta, \gamma, X)$ and uniform normal structure}

In this part, we establish a sufficient condition for uniform normal structure in terms of the constant  $C_{\mathrm{NJ}}^p(\alpha,\beta,\gamma,X)$ . 

\begin{lemma}\label{le2}
	Let \( X \) be a Banach space and \( \widetilde{X} \) be its ultrapower coincide. Then 
	$C_{\mathrm{NJ}}^p(\alpha, \beta, \gamma, X) = C_{\mathrm{NJ}}^p(\alpha, \beta, \gamma, \widetilde{X})$ for all \( \alpha, \beta, \gamma > 0 \) and \( p \geq 1 \).
\end{lemma}

\begin{proof}
	First, it is clear that \( C_{\mathrm{NJ}}^p(\alpha, \beta, \gamma, X) \leq C_{\mathrm{NJ}}^p(\alpha, \beta, \gamma, \widetilde{X}) \). We now show \( C_{\mathrm{NJ}}^p(\alpha, \beta, \gamma, X) \geq C_{\mathrm{NJ}}^p(\alpha, \beta, \gamma, \widetilde{X}) \).
	
	Let \( \delta > 0 \), and take \( \widetilde{x}, \widetilde{y} \in \widetilde{X} \) with \( (\widetilde{x}, \widetilde{y}) \neq (0, 0) \). Without loss of generality, assume \( \widetilde{x} \neq 0 \) (the case \( \widetilde{y} \neq 0 \) is similar). Choose \( \varepsilon > 0 \) such that \( \varepsilon < \delta \|\widetilde{x}\| \).
	
	Because of \( \widetilde{X} \) is the completion of \( X \), there exist sequences \( \{x_n\}, \{y_n\} \subset X \) such that \( x_n \to \widetilde{x} \) and \( y_n \to \widetilde{y} \).
	Since
	\[
	\begin{aligned}
		a :=& \frac{\alpha^p\|\beta \widetilde{x} + \gamma \widetilde{y}\|^p + \beta^{\frac{p}{2}}\|\gamma \widetilde{x} - \alpha \widetilde{y}\|^p}{2^{p-2}\left[(\alpha\beta)^{\frac{p}{2}} + \gamma^p\right]\left(\beta^{\frac{p}{2}}\|\widetilde{x}\|^p + \alpha^{\frac{p}{2}}\|\widetilde{y}\|^p\right)}\\
		=&\lim_\mathcal{U}\frac{\alpha^p\|\beta x_n + \gamma y_n\|^p + \beta^{\frac{p}{2}}\|\gamma x_n - \alpha y_n\|^p}{2^{p-2}\left[(\alpha\beta)^{\frac{p}{2}} + \gamma^p\right]\left(\beta^{\frac{p}{2}}\|x_n\|^p + \alpha^{\frac{p}{2}}\|y_n\|^p\right)}\\
		:=&\lim_\mathcal{U}a_n,
	\end{aligned}		
	\]
	the set \( \{ n \in \mathbb{N} : |a_n - a| < \delta \text{ and } \|x_n - \widetilde{x}\| < \varepsilon, \|y_n - \widetilde{y}\| < \varepsilon \} \)  belongs $\mathcal{U}$. In particular, for some \( n \), we have
	\[a < a_n + \delta \leq C_{\mathrm{NJ}}^p(\alpha, \beta, \gamma, X) + \delta.\]
	
	Then, we have \( C_{\mathrm{NJ}}^p(\alpha, \beta, \gamma, \widetilde{X}) \leq C_{\mathrm{NJ}}^p(\alpha, \beta, \gamma, X) \).
\end{proof}

In \cite{40}, we have
a Banach space \( X \) has uniform normal structure if there exists some \( t \in (0, 1] \) such that the inequality $ C_{\mathrm{NJ}}^{(p)}(X) < \frac{(\sqrt{4 + t^2} + t)^p}{2^{2p - 2}(1 + t^p)}$ holds for $p \geq 1$.

The Theorem~\ref{t6} covers this result by setting $\alpha=\beta=\gamma$, under which our generalized constant $C_{\mathrm{NJ}}^p(\alpha, \alpha, \alpha, X)$ reduces to
the classical $p$-th von Neumann--Jordan constant $C_{\mathrm{NJ}}^{(p)}(X)$, and the right-
hand side of the inequality in Theorem\cite{40} simplifies exactly to the bound given in \cite{40}.
Thus, the condition for uniform normal structure in \cite{40} is a special case of Theorem~\ref{t6},
demonstrating that our result is more general.

\begin{theorem}\label{t6}
	Let \(\alpha \geq \gamma \geq \alpha t\), \(X\) be a Banach space and the inequality
	\[
	\begin{aligned}
		&C_{\mathrm{N J}}^p(\alpha,\beta,\gamma,X)\\ <& \frac{\alpha^{\frac{p}{2}}\gamma^p \left( \beta\sqrt{4\alpha^2 + \gamma^2 t^2} + \beta\gamma t \right)^p + \alpha^p \beta^{\frac{p}{2}} \left( \beta\sqrt{4\alpha^2 + \gamma^2 t^2} - 2\alpha\beta + \beta\gamma t + 2\gamma^2 \right)^p}{2^{2p-2} \alpha^p \gamma^p \left[ \alpha^{\frac{p}{2}} \left( \beta^p + \gamma^p t^p \right) + \beta^{\frac{p}{2}} \left( \gamma^p + \alpha^p t^p \right) \right]}
	\end{aligned}
	\]
	holds for some \(t \in (0,1]\) and $p \geq 1$. Then \(X\) has uniform normal structure.
\end{theorem}

\begin{proof}
	Combine Corollary \ref{cor1}, it is easy to see that \(X\) is uniformly non-square, hence \(X\) is super-reflexive (see \cite{15}). Therefore, it is sufficient to prove that \(X\) has normal structure. Applying Lemma 2 in \cite{16}, there exist \(\tilde{x}_{1}, \tilde{x}_{2}, \tilde{x}_{3} \in S_{\tilde{X}}\) and \(\tilde{f}_{1}, \tilde{f}_{2}, \tilde{f}_{3} \in S_{\tilde{X}^{*}}\) satisfying the following three properties:
	\begin{enumerate}
		\item[(i)] \(\|\tilde{x}_{i} - \tilde{x}_{j}\| = 1\) and \(\tilde{f}_{i}(\tilde{x}_{j}) = 0\) for all \(i \neq j\),
		\item[(ii)] \(\tilde{f}_{i}(\tilde{x}_{i}) = 1\) for \(i = 1, 2, 3\),
		\item[(iii)] \(\|\tilde{x}_{3} - (\tilde{x}_{2} + \tilde{x}_{1})\| \geq \|\tilde{x}_{2} + \tilde{x}_{1}\|\).
	\end{enumerate}
	
	Next, let 
	\[
	g(t) = \frac{\sqrt{4\alpha^{2} + \gamma^{2}t^{2}} + 2\alpha - \gamma t}{2\beta},  \quad t\in(0,1]
	\]
	and consider three possible cases.
	
	\textbf{Case 1:} If \(\|\tilde{x}_{1} + \tilde{x}_{2}\| \leq g(t)\). Let \(\tilde{x} = \tilde{x}_{1} - \tilde{x}_{2}\) and \(\tilde{g} = \dfrac{\tilde{x}_{1} + \tilde{x}_{2}}{g(t)}\). Then \(\tilde{x}, \tilde{g} \in B_{\tilde{X}}\), we have
	\[
	\begin{aligned}
		& \|\beta \tilde{x} + \gamma t \tilde{y}\| = \left\| \left( \beta + \frac{\gamma t}{g(t)} \right) \tilde{x_1} - \left( \beta - \frac{\gamma t}{g(t)} \right) \tilde{x_2} \right\| \\
		\geq& \left( \beta + \frac{\gamma t}{g(t)} \right) f_1(\tilde{x_1}) - \left| \beta - \frac{\gamma t}{g(t)} \right| f_1(\tilde{x_2}) \\
		=& \beta + \frac{\gamma t}{g(t)} 
	\end{aligned}
	\]
	and
	\[
	\begin{aligned}
		& \|\gamma \tilde{x} - \alpha t \tilde{y}\| = \left\| \left( \gamma - \frac{\alpha t}{g(t)} \right) \tilde{x_1} - \left( \gamma + \frac{\alpha t}{g(t)} \right) \tilde{x_2} \right\| \\
		\geq& \left( \gamma + \frac{\alpha t}{g(t)} \right) f_2(\tilde{x_2}) - \left| \gamma - \frac{\alpha t}{g(t)} \right| f_2(\tilde{x_1}) \\
		=& \gamma + \frac{\alpha t}{g(t)}.
	\end{aligned}
	\]
	\textbf{Case 2:} If $\|\tilde{x}_{1}+\tilde{x}_{2}\|\geq g(t)$ and $\|\tilde{x}_{3}+\tilde{x}_{2}-\tilde{x}_{1}\|\leq g(t)$. Let $\tilde{x}=\tilde{x}_{2}-\tilde{x}_{3}$ and $\tilde{y}=(\tilde{x}_{3}+\tilde{x}_{2}-\tilde{x}_{1})$. Then $\tilde{x},\tilde{y}\in B_{\tilde{X}}$, we have
	\[
	\begin{aligned}
		& \|\beta \tilde{x} + \gamma t \tilde{y}\| = \left\| \left( \beta + \frac{\gamma t}{g(t)} \right) \tilde{x_1} - \left( \beta - \frac{\gamma t}{g(t)} \right) \tilde{x_3} - \frac{\gamma t}{g(t)} \tilde{x_1} \right\| \\
		\geq& \left( \beta + \frac{\gamma t}{g(t)} \right) \tilde{f}_2(\tilde{x_1}) - \left( \beta - \frac{\gamma t}{g(t)} \right) \tilde{f}_2(\tilde{x_3}) - \frac{\gamma t}{g(t)} \tilde{f}_2(\tilde{x_1}) \\
		=& \beta + \frac{\gamma t}{g(t)} 
	\end{aligned}
	\]
	and
	\[
	\begin{aligned}
		\|\gamma\tilde{x}-\alpha t\tilde{y}\| &= \left\|\left(\gamma+\frac{\alpha t}{g(t)}\right)\tilde{x}_{3}+\left( \gamma-\frac{\alpha t}{g(t)}\right)\tilde{x}_{2}-\frac{\alpha t}{g(t)}\tilde{x}_{1} \right\| \\
		&\geq \left(\gamma+\frac{\alpha t}{g(t)}\right)\tilde{f}_{3}(\tilde{x}_{3})-\left|\gamma-\frac{\alpha t}{g(t)}\right|\tilde{f}_{3}(\tilde{x}_{2})-\frac{ \alpha t}{g(t)}\tilde{f}_{3}(\tilde{x}_{1}) \\
		&= \gamma+\frac{\alpha t}{g(t)}.
	\end{aligned}
	\]
	\textbf{Case 3:} If $\|\tilde{x}_{1}+\tilde{x}_{2}\|\geq g(t)$ and $\|\tilde{x}_{3}+\tilde{x}_{2}-\tilde{x}_{1}\|\geq g(t)$. Let $\tilde{x}=\tilde{x}_{3}-\tilde{x}_{1}$ and $\tilde{y}=\tilde{x}_{2}$. Then $\tilde{x},\tilde{y}\in B_{\tilde{X}}$, we have
	\[
	\begin{aligned}
		\|\beta \tilde{x} + \gamma t \tilde{y}\| &= \|\beta (\tilde{x_3} - \tilde{x_1}) + \gamma t \tilde{x_2}\| \\&= \|\beta (\tilde{x_3} + \tilde{x_2} - \tilde{x_1}) + (\gamma t - \alpha \beta) \tilde{x_2}\| \\
		& \geq \beta (g(t) - 1) + \gamma t \\
	\end{aligned}
	\]
	and
	\[
	\begin{aligned}
		\|\gamma \tilde{x} - \alpha t \tilde{y}\| &= \|\gamma (\tilde{x_3} - \tilde{x_1}) - \alpha t \tilde{x_2}\| \\&= \|\gamma (\tilde{x_3} - (\tilde{x_1} + \tilde{x_2})) + (\gamma - \alpha t) \tilde{x_2}\| \\
		& \geq \gamma (g(t) - 1) + \alpha t.
	\end{aligned}
	\]
	Then, based on the definition of \( C_{\mathrm{NJ}}^p(\alpha, \beta, \gamma, X) \) and the result \( C_{\mathrm{NJ}}^p(\alpha, \beta, \gamma, X) = C_{\mathrm{NJ}}^p(\alpha, \beta, \gamma, \widetilde{X}) \) from Lemma \ref{le2}, we obtain
	\[
	\begin{aligned}
		&C_{\mathrm{N J}}^p(\alpha,\beta,\gamma,X) \\\geq& \max\left\{\frac{\alpha^{\frac{p}{2}} \left\| \beta + \frac{\gamma t}{g(t)} \right\|^p + \beta^{\frac{p}{2}} \left\| \gamma + \frac{\alpha t}{g(t)} \right\|^p}{2^{p - 2} \left[ \alpha^{\frac{p}{2}} (\beta^p + \gamma^p t^p) + \beta^{\frac{p}{2}} (\gamma^p + \alpha^p t^p) \right]}\right.,\\
		&\quad \left. \frac{\alpha^{\frac{p}{2}} \left\| \beta (g(t) - 1) + \gamma t \right\|^p + \beta^{\frac{p}{2}} \left\| \gamma (g(t) - 1) + \alpha t \right\|^p}{2^{p - 2} \left[ \alpha^{\frac{p}{2}} (\beta^p + \gamma^p t^p) + \beta^{\frac{p}{2}} (\gamma^p + \alpha^p t^p) \right]} \right\}\\
		=&\frac{\alpha^{\frac{p}{2}}\gamma^p \left( \beta\sqrt{4\alpha^2 + \gamma^2 t^2} + \beta\gamma t \right)^p + \alpha^p \beta^{\frac{p}{2}} \left( \beta\sqrt{4\alpha^2 + \gamma^2 t^2} - 2\alpha\beta + \beta\gamma t + 2\gamma^2 \right)^p}{2^{2p-2} \alpha^p \gamma^p \left[ \alpha^{\frac{p}{2}} \left( \beta^p + \gamma^p t^p \right) + \beta^{\frac{p}{2}} \left( \gamma^p + \alpha^p t^p \right) \right]}.
	\end{aligned}
	\]
	This is a contradiction. The proof is complete.
\end{proof}

\begin{remark}
	In \cite{101}, Jiaye Bi et al. established the operator-level von Neumann-Jordan constant $\bar{C}_{\mathrm{NJ}}(T),$ which is defined as following:
	\[\bar{C}_{\mathrm{NJ}}(T)=\sup \left\{\frac{\| Tx+ y\|^2+\| Tx- y\|^2}{\left(2\|x\|^2+2\|y\|^2\right)}: x, y \in X,(x, y) \neq(0,0)\right\}.\]
	
	Inspired by this work, we can further extend the generalized von Neumann-Jordan type constant $C_{\mathrm{NJ}}^{p}(\alpha, \beta, \gamma, X)$ established in this paper to the operator level. 
	
	Let $X$ be a Banach space and $T:X \to X$ be a bounded linear operator. For $\alpha,\beta,\gamma >0$ and $p \geq 1$, we define
	\[
	\bar{C}_{\mathrm{NJ}}^p(\alpha, \beta, \gamma, T)=\sup \left\{\frac{\alpha^{\frac{p}{2}}\|\beta Tx+\gamma y\|^p+\beta^{\frac{p}{2}}\|\gamma Tx-\alpha y\|^p}{2^{p-2}\left[(\alpha \beta)^{\frac{p}{2}}+\gamma^p\right]\left(\beta^{\frac{p}{2}}\|x\|^p+\alpha^{\frac{p}{2}}\|y\|^p\right)}: x, y \in X,(x, y) \neq(0,0)\right\} .
	\]

	It is easy to see that when $\alpha=\beta=\gamma$ and $p=2$, $\overline{C}_{\mathrm{NJ}}^{p}(\alpha, \beta, \gamma, T)$ exactly reduces to the operator-level von Neumann-Jordan constant $\bar{C}_{\mathrm{NJ}}(T)$ in \cite{101}. We also establish the  bounds for this operator-type constant:
	\[
	\max\left\{ \frac{1}{2^{p-2}}, \frac{\|T\|^p}{2^{p-2}} \right\} \leq \bar{C}_{\mathrm{NJ}}^p(\alpha, \beta, \gamma, T) \leq 2(1+\|T\|^p).
	\]
\end{remark}
\begin{proof}
	On the one hand, for any $x \in S_X$ and $y=0$, we have 
	\[
	\overline{C}_{\mathrm{NJ}}^p(\alpha, \beta, \gamma, T) \ge \frac{\alpha^{\frac{p}{2}} \|\beta x + 0\|^p + \beta^{\frac{p}{2}} \|\gamma x - 0\|^p}{2^{p-2} \left[ (\alpha\beta)^{\frac{p}{2}} + \gamma^p \right] (\beta^{\frac{p}{2}} \|x\|^p + 0)} = \frac{1}{2^{p-1}}.
	\]
	
	For any $\varepsilon > 0$, there exists $y \in S_X$ such that $\| Ty \| \ge \frac{\| T \|}{1+\varepsilon}$.
	Then, we have
	\[
	\overline{C}_{\mathrm{NJ}}^p(\alpha, \beta, \gamma, T) \ge \frac{\alpha^{\frac{p}{2}} \|0 + \gamma Ty\|^p + \beta^{\frac{p}{2}} \|0 - \alpha Ty\|^p}{2^{p-2} \left[ (\alpha\beta)^{\frac{p}{2}} + \gamma^p \right] \cdot (\alpha^{\frac{p}{2}} + 0)} = \frac{\|Ty\|^p}{2^{p-2}} \ge \frac{1}{2^{p-2}} \left( \frac{\|T\|}{1+\varepsilon} \right)^p.
	\]
	Then, we obtain $\overline{C}_{\mathrm{NJ}}^p(\alpha, \beta, \gamma, T) \ge \frac{1}{2^{p-2}}$.
	
	Thus,
	\[
	\overline{C}_{\mathrm{NJ}}^p(\alpha, \beta, \gamma, T) \ge \max\left\{ \frac{1}{2^{p-2}}, \frac{\|T\|^p}{2^{p-2}} \right\}.
	\]
	
	On the other hand, we have
	\[
	\begin{aligned}
		\| \beta x + \gamma Ty \|^p \le (\beta\|  x \| + \gamma\| T\| \|y \|)^p &\stackrel{\text{Lamma \ref{le1}}}{\le} 2^{p-1} \left[ \beta^p \| x \|^p + \gamma^p \| T\|^p \|y \|^p \right] \\ &\le 2^{p-1} [1+\|T\|^p] \left[ \beta^p \| x \|^p + \gamma^p \| y \|^p \right]
	\end{aligned}
	\]
	and
	\[
	\| \gamma x - \alpha Ty \|^p \le 2^{p-1} [1+\|T\|^p] \left[ \gamma^p \| x \|^p + \alpha^p \| y \|^p \right].
	\]
	Then,we can obtain
	\[
	\begin{aligned}
		& \alpha^{\frac{p}{2}} \| \beta x + \gamma Ty \|^p + \beta^{\frac{p}{2}} \| \gamma x - \alpha Ty \|^p \\
		\le & 2^{p-1} [1+\|T\|^p] \left[ (\alpha \beta)^{\frac{p}{2}} + \gamma^p \right] (\beta^{\frac{p}{2}} \| x \|^p + \alpha^{\frac{p}{2}} \| y \|^p).
	\end{aligned}
	\]
	
	Thus,
	\[
	\overline{C}_{\mathrm{NJ}}^p(\alpha, \beta, \gamma, T) \le \frac{2^{p-1} [1+\|T\|^p] \left[ (\alpha \beta)^{\frac{p}{2}} + \gamma^p \right] (\beta^{\frac{p}{2}} \| x \|^p + \alpha^{\frac{p}{2}} \| y \|^p)}{2^{p-2} \left[ (\alpha \beta)^{\frac{p}{2}} + \gamma^p \right] (\beta^{\frac{p}{2}} \| x \|^p + \alpha^{\frac{p}{2}} \| y \|^p)} \le 2(1+\|T\|^p).
	\]
\end{proof}

For the properties of $\overline{C}_{\mathrm{NJ}}^{p}(\alpha, \beta, \gamma, T)$ on Banach spaces, further research is needed.

Thanks to all the members of the Functional Analysis Research team of the College of Mathematics and Statistics of Anqing Normal University for their discussion and correction of the difficulties and errors encountered in this paper. The authors cordially thank Professor Hai Zhang for his encouraging advice and the first author would like to acknowledge his strong support for undergraduate students in carrying out basic mathematics research.

\bmhead{Author contributions}
All the authors have contributed equally for this work.

\bmhead{Funding}
	This work was supported by Anqing Normal University Innovation and Entrepreneurship Training Program (No. S202610372033).

\bmhead{Data availability}
No data was used for the research described in the article.

\section*{Declarations}
\bmhead{Conflict of interest statement}

The authors declare that they have no known competing financial interests or personal relationships that could have appeared to influence the work reported in this paper.

\bmhead{Ethical approval}
This article does not contain any studies with human participants or animals performed by any of the authors.

\end{document}